\documentclass[secnum,leqno]{rims27-kb}
\usepackage{amssymb, amsmath}%%% option packages
\newtheorem{lmm}{Lemma}[section]

\newtheorem{prp}{Proposition}[section]

\newtheorem{thm}{Theorem}[section]

\theoremstyle{definition}
\newtheorem{dfn}{Definition}[section]
\theoremstyle{remark}
\newcommand{\dl}{\delta}
\newcommand{\eps}{\varepsilon}
\newcommand{\lm}{\lambda}
\newcommand{\sg}{\sigma}
\newcommand{\om}{\omega}
\newcommand{\Gm}{\Gamma}
\newcommand{\Om}{\Omega}
\newcommand{\D}{{\mathbb D}}
\newcommand{\N}{{\mathbb N}}
\newcommand{\Q}{{\mathbb Q}}
\newcommand{\R}{{\mathbb R}}
\newcommand{\cE}{{\mathcal E}}
\newcommand{\cF}{{\mathcal F}}
\newcommand{\la}{\langle}
\newcommand{\ra}{\rangle}
\newcommand{\cp}{\mathop{\rm Cap}\nolimits}

\title{Reflecting Ornstein-Uhlenbeck processes\\on pinned path spaces}
\TitleHead{Reflecting Ornstein-Uhlenbeck processes}          %optional
\author{Masanori \textsc{Hino}%
          \footnote{Graduate School of Informatics, Kyoto University, Kyoto 606-8501, Japan.}
\ and
Hiroto \textsc{Uchida}%
\footnote{Minato-ku, Tokyo 105-0004, Japan.}}
\AuthorHead{Masanori Hino and Hiroto Uchida}         %optional but recommended
\classification{60J60, 31C25, 28C20.}
\support{The first author is partly supported by the Ministry of Education, Culture, Sports, Science and Technology, Grant-in-Aid for Encouragement of Young Scientists, 18740070.}  %optional
\keywords{{\em Key Words}\/: Ornstein-Uhlenbeck process, Dirichlet forms, BV function, Wiener space.}         %optional
\VolumeNo{4x}           %editors must define this.
\YearNo{200x}           %editors must define this.
\PagesNo{000--000}      %editors must define this.
\communication{KK, June 8, 2000
	Revised September 11, 2000.}      %editors must define this.
\begin{document}
%
% The text goes here.  
% Be sure to use the appropriate "theorem-like" environment as 
% is the following examples.  Never use plain TeX commands for these, as
% they will cause interference with the styles of other papers. 

\maketitle

%\tableofcontents      %optional

\begin{abstract}     
Consider a set of continuous maps from the interval $[0,1]$ to a domain in ${\mathbb R}^d$. Although the topological boundary of this set in the path space is not smooth in general, by using the theory of functions of bounded variation (BV functions) on the Wiener space and the theory of Dirichlet forms, we can discuss the existence of the surface measure and the Skorokhod representation of the reflecting Ornstein-Uhlenbeck process associated with the canonical Dirichlet form on this set. 
\end{abstract}

\section{Introduction}
In \cite{Ha}, Hariya obtained an integration by parts formula on a subset of the pinned path space on $\R^d$, which is a partial generalization of the work by Zambotti~\cite{Za}. 
To state it more precisely, let $\Om$ be a bounded domain in $\R^d$. We assume that the boundary of $\Om$ is sufficiently smooth.
Take $a,b\in \Om$ and define the path spaces as follows:
\begin{align*}
W_{a,b}&=\{w\in C([0,1]\to\R^d)\mid w(0)=a,\ w(1)=b\},\\
W_{a,b}^\Om&=\{w\in C([0,1]\to \Om)\mid w(0)=a,\ w(1)=b\},\\
H_0&=\left\{h\in C([0,1]\to\R^d)\left|\, \begin{array}{l}h(0)=h(1)=0,\ \text{$h$ is absolutely continuous}\\
\text{and }\int_0^1 |\dot h(s)|_{\R^d}^2\,ds<\infty\end{array}\right.\right\}.
\end{align*}
We regard $W_{a,b}^\Om$ as a subset of $W_{a,b}$.
The topological boundary $\partial W_{a,b}^\Om$ of $W_{a,b}^\Om$ with respect to the uniform topology is given by
\[
  \partial W_{a,b}^\Om=\left\{w\in W_{a,b}\left|
  \begin{array}{l}
  w(t)\in\overline{\Om}\text{ for every $t\in[0,1]$ and}\\
  w(s)\in\partial \Om\text{ for some $s\in(0,1)$}
  \end{array}
  \right.\right\},
\]
where $\overline{\Om}$ and $\partial \Om$ denote the closure and the boundary of $\Om$ in $\R^d$, respectively.
We define a subset $\partial' W_{a,b}^\Om$ of $\partial W_{a,b}^\Om$ by
\[
\partial' W_{a,b}^\Om=\{w\in \partial W_{a,b}^\Om\mid{}\text{there exists a unique }s\in(0,1) \text{ such that } w(s)\in\partial\Om\}.
\]
Let $\mu_{a,b}$ be the pinned Wiener measure on $W_{a,b}$.
For a smooth cylindrical function $F$ on $W_{a,b}$ and $h\in H_0$, Hariya~\cite{Ha} proved the identity
\begin{equation}\label{eq:ibp}
  \int_{W_{a,b}^\Om} \partial_h F(w)\,\mu_{a,b}(dw)
  =\int_{W_{a,b}^\Om}F(w)\la h,w\ra\,\mu_{a,b}(dw)
  + ({\rm BC}),
\end{equation}
where $\partial_h$ is the partial derivative along the direction of $h$, $\la h,w\ra$ denotes the Wiener integral $\int_0^1 h(s)\,dw(s)$, and (BC) is the ``boundary contribution,'' which is expressed as an integral over $\partial W_{a,b}^\Om$.
The explicit expression of (BC) is provided in \cite{Ha}.
In this study, we provide only the following remarks on (BC).
\begin{enumerate}
\item[(a)] The mass of the measure on $\partial W_{a,b}^\Om$ appearing in the integral representation of (BC) concentrates on $\partial' W_{a,b}^\Om$.
\item[(b)] The integrand in (BC) contains the normal derivatives of the heat kernel density on $\Om$ at $\partial \Om$ with the Dirichlet boundary condition.
\end{enumerate}
The integration by parts formula \eqref{eq:ibp} implies that the indicator function $1_{W_{a,b}^\Om}$ of $W_{a,b}^\Om$ is a BV function, and we can construct the reflecting Ornstein-Uhlenbeck process on $W_{a,b}^\Om$ with the Skorokhod representation (cf.\ Section~2 below).
On the other hand, property (b) above imposes on the strong regularity of $\Om$ since we cannot expect the normal derivatives of the heat kernel density to exist at the boundary if $\partial \Om$ is not very smooth.
%(According to \cite{Hs}, if $\partial \Om$ is in the $C^3$-class, all the assumptions in \cite{Ha} for \eqref{eq:ibp} to hold are satisfied.)
If we are only interested in the probabilistic aspect, it is sufficient to prove that $1_{W_{a,b}^\Om}$ is a BV function; in other words, even if we do not know the explicit expression of (BC), only proving the validity of the integration by parts is sufficient.
This is expected to be done under a milder assumption on $\Om$ since such a claim can be proved only by a series of inequalities and not by equalities.
This is the objective of this paper.

In this paper, we introduce the concept of the {\em uniform exterior ball condition} for $\Om$, which allows some singularity at $\partial\Om$, and prove that $1_{W_{a,b}^\Om}$ is a BV function under such a condition.
Based on this, we can construct the reflecting Ornstein-Uhlenbeck process on the closure of $W_{a,b}^\Om$ and prove its Skorokhod representation. 
Further, we prove that the mass of the measure on $\partial W_{a,b}^\Om$ appearing in the Skorokhod representation concentrates on $\partial'W_{a,b}^\Om$, which is consistent with  property~(a) above.
The proof is based on the quantitative estimates of Brownian motion on $\R^d$, and the method is different from that used in \cite{Ha}. 
We expect that our method is sufficiently flexible to discuss more general situations.

This paper is organized as follows. In Section~2, we provide a framework and state the main theorem.
Some key estimates for the Brownian motion on $\R^d$ are proved in Section~3. These estimates are obtained by reducing them to a few detailed estimates of a one-dimensional Brownian motion with a constant drift.
The main theorem is proved in Section~4.
In the last section, we provide a few remarks.

\section{Framework and the main result}
First, we recall the concept of the BV functions on the Wiener space, according to \cite{FH}.
Let $(E,H,\mu)$ be an abstract Wiener space,
that is, $E$ is a separable Banach space, $H$ is a separable Hilbert
space densely and continuously embedded in $E$, and $\mu$ is a
Gaussian measure on $E$ that satisfies the condition
\[
  \int_E \exp\left(\sqrt{-1}\,l(z)\right)\mu(dz)=\exp\left(-|l|_H^2/2\right),
  \quad l\in E^*.
\]
Here, $^*$ denotes the topological dual and we use natural inclusions and an identification $E^*\subset H^*\cong H\subset E$.
When $M$ is a separable Hilbert space, $L^p(E\to M)$ denotes the $L^p$-space on $E$ with respect to $\mu$ which consists of $M$-valued functions.
When $M=\R$, we omit $M$ from the notation.
Let $C_b^1(\R^m)$ be the set of all bounded continuous functions $f$ on $\R^m$ such that all the first-order partial derivatives of $f$ are bounded and continuous.
Define
\begin{align*}
\cF C_b^1&=\left\{ u\colon E\to \R\left|\,
\begin{array}{l}
u(z)=f(l_1(z),\ldots,l_m(z)),\ 
l_1,\ldots,l_m\in E^*,\\
f\in C_b^1(\R^m)\text{ for some $m\in\N$}
\end{array}
\right.\right\},\\
(\cF C_b^1)_{E^*}&=\left\{ G\colon E\to E^*\left|\,
\begin{array}{l}
G(z)=\sum_{i=1}^m u_i(z)l_i,\ 
l_1,\ldots,l_m\in E^*,\\
u_1,\ldots,u_m\in \cF C_b^1\text{ for some $m\in\N$}
\end{array}
\right.\right\}.
\end{align*}
For $u\in\cF C_b^1$, an $H$-valued function $\nabla u$ on $E$ is given by the following identity:
\[
\la \nabla u(z),l\ra_H=\lim_{\eps\to 0}(u(z+\eps l)-u(z))/\eps,
\quad l\in E^*\subset H,\ z\in E.
\]
Let $\nabla^*$ be a (formal) adjoint operator of $\nabla$, which is defined by the following relation:
\[
\la \nabla^* G,u\ra_{L^2(E)}=\la G,\nabla u\ra_{L^2(E\to H)},
\quad u\in \cF C_b^1.
\]
We define the set of BV functions on $E$ as
\[
BV(E)=\left\{\rho\colon E\to \R\left|\,
\begin{array}{l}
\int_E |\rho|\max\{0,\log |\rho|\}^{1/2}\,d\mu<\infty \text{ and there exists}\\
C\ge0\text{ such that }
\left|\int_E (\nabla^* G)\rho\,d\mu\right|
\le C\| |G|_H\|_{L^\infty(E)}\\
\text{for all }G\in (\cF C_b^1)_{E^*}
\end{array}
\right.\!\right\}.
\]
We shall now revisit several properties of BV functions on $E$.
\begin{thm}[({\cite[Theorems~3.7, 3.9]{FH}})] \label{th:1}
For $\rho\in L^1(E)$, the following are equivalent conditions.
\begin{enumerate}
\item $\rho\in BV(E)$.
\item There exists a sequence $\{\rho_n\}$ in $\D^{1,1}:=\overline{\cF C_b^1}^{\|\nabla\cdot\|_{L^1(E\to H)}+\|\cdot\|_{L^1(E)}}$ such that $\rho_n$ converges to $\rho$ in $L^1(E)$ and $\|\nabla \rho_n\|_{L^1(E\to H)}$ is bounded in $n$.
\item (Integration by parts formula) There exist, a finite Borel measure $\nu$ and an $H$-valued function $\sg$ on $E$ such that $|\sg|_H=1$ $\nu$-a.e.\ and
\[
  \int_E (\nabla^* G)\rho\,d\mu=\int_E \la G,\sg\ra_H\,d\nu,\quad
  G\in (\cF C_b^1)_{E^*}.
\]
\end{enumerate}
\end{thm}
\begin{thm}[({\cite[Theorem~4.2]{FH}})] \label{th:2}
Let $\rho\in BV(E)$ and assume $\rho\ge0$ $\mu$-a.e.
Let $S$ be the support of the measure $\rho\,d\mu$.
Define a bilinear form on $L^2(S,\rho\,d\mu)$ by
\[
  \cE(f,g)=\frac12\int_S \la \nabla f,\nabla g\ra_H \rho\,d\mu,\quad
  f,g\in\cF C_b^1.
\]
Assume that $(\cE,\cF C_b^1)$ is closable on $L^2(S,\rho\,d\mu)$.
Then, its closure $(\cE,\cF)$ is a quasi-regular, local, and conservative Dirichlet form on $L^2(S,\rho\,d\mu)$, and the following Skorokhod representation holds:
\begin{equation}\label{eq:skorohod}
X_t=X_0+B_t-\frac12\int_0^t X_s\,ds+\frac12\int_0^t \sigma(X_s)\,dA_s,\quad
t\ge0,\ P_w\text{-a.e.\ for q.e.\ $w$}.
\end{equation}
Here, $(X_t,P_w)$ is a diffusion process on $S$ associated with $(\cE,\cF)$, $\{B_t\}$ is an $E$-valued Brownian motion starting at $0$, $\{A_t\}$ is an additive functional in Revuz correspondence with $\nu$, and $\nu$ and $\sg$ are provided in Theorem~$\ref{th:1}$~$(3)$.
\end{thm}
Note that $\nu$ above is smooth with respect to the $(\cE,\cF)$ from \cite[Theorem~3.9]{FH}, which justifies the consideration of the Revuz correspondence of $\nu$.
When $\rho$ is an indicator function, we term $\{X_t\}$ a reflecting Ornstein-Uhlenbeck process on $S$.
In such a case, the measure $\nu$ can be regarded as a surface measure of $S$.
\begin{thm}[({\cite[Theorem~3.15]{FH}})]\label{th:3}
Under the conditions described in Theorem~$\ref{th:2}$, if moreover $\rho$ is an indicator function of a set $U$, then the mass of $\nu$ concentrates on the topological boundary of $U$.
\end{thm}
We remark that the original assertion of Theorem~3.15 in \cite{FH} provides more detailed information on the support of $\nu$.

For $t>0$, $x,y\in\R^d$, we define 
\[
p_t(x,y)=(2\pi t)^{-d/2}\exp\left(-\frac{|x-y|_{\R^d}^2}{2t}\right).
\]
Fix $a,b\in\R^d$, and let $W_{a,b}$ and $H_0$ as defined in Section~1.
The pinned Wiener measure $\mu_{a,b}$ on $W_{a,b}$ is a Borel probability measure such that for $0=t_0<t_1<\cdots<t_{N+1}=1$ and Borel sets $A_1,\ldots,A_N$ of $\R^d$,
\begin{align*}
&\mu_{a,b}[w\in W_{a,b}\mid  w_{t_i}-w_{t_{i-1}}\in A_i,\ i=1,\ldots,N]\\
&=p_1(a,b)^{-1}\int\!\cdots\!\int_{A_1\times \dots\times A_N}\prod_{i=1}^{N+1}p_{t_i-t_{i-1}}(x_{i-1},x_i)\,dx_1\cdots dx_N,
\end{align*}
where $x_0=a$ and $x_{N+1}=b$.
Then, $(W_{0,0},H_0,\mu_{0,0})$ is an abstract Wiener space.
When $(a,b)\ne (0,0)$, $W_{a,b}$ is not a linear space.
However, $W_{a,b}$ is isomorphic to $W_{0,0}$ as an  affine space according to the shift map \[\lm_{a,b}\colon W_{a,b}\ni w\mapsto w-h_{a,b}\in W_{0,0},\] where $h_{a,b}(t)=a+(b-a)t$, $t\in[0,1]$, and $(W_{a,b},\mu_{a,b})$ is isomorphic to $(W_{0,0},\mu_{0,0})$ as a measure space according to the map $\lm_{a,b}$.
Therefore, by pushing everything forward to $(W_{0,0},\mu_{0,0})$, we can define the concepts of $\cF C_b^1$, $\nabla$, the BV space $BV(W_{a,b})$ etc., on $(W_{a,b},H_0,\mu_{a,b})$.
Furthermore, Theorems~\ref{th:1}, \ref{th:2}, and \ref{th:3} are valid on this space by appropriate modification.

Let $\Om$ be a domain of $\R^d$.
We do not assume that $\Om$ is bounded, but assume $\Om\ne \R^d$.
For $x\in \R^d$ and $r>0$, $B(x,r)$ denotes the closed ball in $\R^d$ with center $x$ and radius $r$.
\begin{dfn}\label{def:uebc}
We state that $\Om$ satisfies the {\em uniform exterior ball condition} if there exists $\dl>0$ such that for every $y\in\partial \Om$, there exists $z\in\R^d\setminus \Om$ satisfying ${B}(z,\dl)\cap \overline{\Om}=\{y\}$. 
\end{dfn}
For example, bounded domains with boundaries in the $C^2$-class and convex domains satisfy the uniform exterior ball conditions. 
It may be said that this condition allows outward cusps, but not inward cusps.

We consider $W_{a,b}^\Om$, $\partial W_{a,b}^\Om$, and $\partial'W_{a,b}^\Om$ as defined in Section~1.
Let $\overline{W_{a,b}^\Om}=W_{a,b}^\Om\cup\partial W_{a,b}^\Om$.
The main theorem in this paper is as follows.
\begin{thm}\label{th:main}
Assume that $\Om$ satisfies the uniform exterior ball condition.
Then, $1_{\overline{W_{a,b}^\Om}}\in BV(W_{a,b})$.
Further, the bilinear form $(\cE',\cF C_b^1)$ on $L^2(\overline{W_{a,b}^\Om},\mu_{a,b}|_{\overline{W_{a,b}^\Om}})$ defined by 
\[
\cE'(f,g)=\frac12\int_{\overline{W_{a,b}^\Om}}\la \nabla f,\nabla g\ra_{H_0}\,d\mu_{a,b},
\quad f,g\in\cF C_b^1
\]
 is closable, and its closure $(\cE',\cF')$ is a quasi-regular, local, and conservative Dirichlet form.
Moreover, when $(X'_t,P'_w)$ denotes the diffusion process on $\overline{W_{a,b}^\Om}$ associated with $(\cE',\cF')$,
$(X_t,P_w):=(\lm_{a,b}(X'_t),P'_{\lm_{a,b}^{-1}(w)}\circ\lm_{a,b}^{-1})$ satisfies the Skorokhod representation~\eqref{eq:skorohod} with $(E,H,\mu)=(W_{0,0},H_0,\mu_{0,0})$ and $\rho=1_{\lm_{a,b}\left(\overline{W_{a,b}^\Om}\right)}$.
Furthermore, $\partial W_{a,b}^\Om\setminus\partial' W_{a,b}^\Om$ has a null capacity that is associated with $(\cE',\cF')$.
In particular, the mass of the measure $\nu$ that  corresponds to $\rho$ in Theorem~$\ref{th:1}$~$(3)$ concentrates on $\lm_{a,b}(\partial' W_{a,b}^\Om)$.
\end{thm}
\section{Some estimates for (pinned) Brownian motion}
Subsequently, $C_i$ denotes an insignificant positive constant and a domain $\Om$ in $\R^d$ is assumed to satisfy the uniform exterior ball condition.

We define a Lipschitz function $q$ on $\R^d$ by 
\[
q(x)=\inf_{y\in\R^d\setminus\Om}|x-y|_{\R^d}-\inf_{y\in\Om}|x-y|_{\R^d}.
\]
For $r\ge0$, set $\Om_r=\{x\in\R^d\mid q(x)>r\}$.
Note that $\Om_0=\Om$ and $\{q(x)\ge0\}=\overline{\Om}$.

Let $W=C([0,\infty)\to\R^d)$.
Let $\{\hat P_x\}_{x\in\R^d}$ be the probability measures on $W$ such that the coordinate process $\{\om_t\}_{t\ge0}$ is a $d$-dimensional Brownian motion starting at $x$ under $\hat P_x$ for each $x\in\R^d$.
For $t\ge0$, let $\hat\cF_t$ be a $\sg$-field generated by $\{\{\om_s\in D\}; s\in[0,t],\ D \text{ is a Borel set of }\R^d\}$.
Then, $\{\hat\cF_t\}$ is a minimal filtration to which $\{\om_t\}$ is adapted on the canonical measurable space $(W,\hat\cF_\infty)$.
For an $\{\hat\cF_t\}$-stopping time $\tau$, define 
$\hat\cF_\tau=\{A\in\hat\cF_\infty\mid A\cap\{\tau\le t\}\in\hat\cF_t \text{ for all }t\ge0\}$.
We denote the integral with respect to $\hat P_x$ by $\hat E_x$.
The shift operator $\theta_s\colon W\to W$ is defined by $(\theta_s \om)_t=\om_{s+t}$, $t\ge0$.
\begin{lmm}\label{lem:ito}
Let $x\in\overline\Om$. 
Choose $y\in\partial\Om$ and $z\in\R^d\setminus\Om$ such that $q(x)=|x-y|_{\R^d}$ and $B(z,\dl)\cap\overline\Om=\{y\}$, where $\dl$ is provided in Definition~$\ref{def:uebc}$.
Let $K=(d-1)/(2\dl)$ and $R_t=|\om_t-z|_{\R^d}$ for $\om=\{\om_t\}\in W$.
Then, for each $u>0$,
\[
\{R_t\ge\dl \text{ for all }t\in[0,u]\}\subset\{R_t\le q(x)+\dl+Kt+S_t \text{ for all }t\in[0,u]\}
\]
up to a $\hat P_x$-null set.
Here, $S_t$ is a one-dimensional Brownian motion under $\hat P_x$ starting at $0$ that is defined by
\[
   S_t(\om)=\sum_{i=1}^d\int_0^t\frac{\om_s^{(i)}-z^{(i)}}{R_s}\,d\om^{(i)}_s,\quad
   \om_s=(\om^{(1)}_s,\ldots,\om^{(d)}_s),\ 
   z=(z^{(1)},\ldots,z^{(d)}).
\]
\end{lmm}
\begin{proof}
Define an $\{\hat\cF_t\}$-stopping time $\sg$ by $\sg=\inf\{t\ge0\mid R_t=0\}$.
Note that $R_0=|x-z|_{\R^d}=q(x)+\dl$ $\hat P_x$-a.e.
By virtue of It\^o's formula,
\[
  R_t=q(x)+\dl+\int_0^t\frac{d-1}{2R_s}\,ds+S_t
  \quad\text{on }\{t<\sg\}\quad
  \hat P_x\text{-a.e.}
\]
Therefore, the assertion holds.
\end{proof}
\begin{prp}\label{prop:1}
There exists $C_1>0$ such that for every $u>0$ and $x\in\overline\Om$,
\[
  \hat P_x\left[\inf_{t\in[0, u]}q(\om_t)\ge0\right]\le C_1(1+u^{-1/2})q(x).
\]
\end{prp}
\begin{proof}
We retain the notations in Lemma~\ref{lem:ito}, from which
\[
\hat P_x\left[ R_t\ge\dl\text{ for all }t\in[0,u]\right]
\le \hat P_x\left[q(x)+\dl+K t+S_t\ge\dl\text{ for all }t\in[0,u]\right].
\]
Let $r>q(x)$ and define $\eta=\inf\{t\ge0\mid Kt+S_t\le -r\}$.
The law of $\eta$ under $\hat P_x$ is given by 
\[
\hat P_x[\eta\in dt]=1_{(0,\infty)}(t)\frac{r}{\sqrt{2\pi t^3}}\exp\left(-\frac{(r+Kt)^2}{2t}\right)\,dt+(1-\exp(-2Kr))\dl_{\infty}(dt),
\]
where $\dl_\infty$ is a delta measure at $\infty$. (See, e.g., \cite[p.~295]{BS}.)
Then, we have
\begin{align*}
\hat P_x\left[\inf_{t\in[0, u]}q(\om_t)\ge0\right]
&\le \hat P_x\left[ R_t\ge\dl\text{ for all }t\in[0,u]\right]
\le \hat P_x[\eta>u]\\
&=\int_u^\infty \frac{r}{\sqrt{2\pi t^3}}\exp\left(-\frac{(r+Kt)^2}{2t}\right)\,dt
+1-\exp(-2Kr)\\
&\le \int_u^\infty \frac{r}{\sqrt{2\pi t^3}}\,dt
+2Kr
=\sqrt{\frac2\pi}\frac{r}{\sqrt u}+2Kr.
\end{align*}
Letting $r\to q(x)$, we obtain the desired inequality.
\end{proof}
For $r>0$, define an $\{\hat\cF_t\}$-stopping time $\tau_r$ by $\tau_r=\inf\{t\ge0\mid \om_t\not\in\Om_r\}$.
Let $\hat P_x^r$ be the law of $\tau_r$ under $\hat P_x$.
\begin{lmm}\label{lem:differentiable}
$\hat P_x^r([0,t])$ is differentiable in $t$ on $(0,\infty)$ and there exists a constant $C_2>0$ such that
$
\frac{d}{dt}\hat P_x^r([0,t])\le C_2 t^{-1}
$.
The constant $C_2$ is taken independently of $x$, $r$ and $t$.
\end{lmm}
\begin{proof}
  It is sufficient to consider the case that $x\in\Om_r$.
  Let $p^r_t(\cdot,\cdot)$ be the transition density of the Brownian motion of $\Om_r$ killed at $\partial \Om_r$.
  Then, 
  \begin{align*}
    \hat P_x^r([0,t])&=\hat P_x[\tau_r\le t]
    =1-\int_{\Om_r}p^r_t(x,z)\,dz\\
    &=1-\iint_{\Om_r\times\Om_r}p^r_s(x,y)p^r_{t-s}(y,z)\,dy\,dz
  \end{align*}
  for $0<s<t$.
  From \cite[Theorem~6.17]{Ou}, $p^r_{t-s}(y,z)$ is differentiable in $t$ on $(s,\infty)$ for a.e.\ $(y,z)$ and the following estimate holds:
  \begin{align*}
  \frac{d}{dt}\hat P_x^r([0,t])
  &\le \iint_{\Om_r\times\Om_r} p^r_s(x,y)\left|\frac{\partial}{\partial t}p^r_{t-s}(y,z)\right|\,dy\,dz\\
  &\le C_3\iint_{\Om_r\times\Om_r} p^r_s(x,y)(t-s)^{-d/2-1}\exp\left(-\frac{C_4|y-z|_{\R^d}^2}{t-s}\right)\,dy\,dz\\
  &\le C_5 (t-s)^{-1}\int_{\Om_r} p^r_s(x,y)\,dy
  \le C_5 (t-s)^{-1},
  \end{align*}
  where $C_3$, $C_4$ and $C_5$ are taken independently of $x$, $r$, $s$ and $t$. By letting $s=t/2$, we complete the proof.
\end{proof}
\begin{prp}\label{prop:path}
There exists $C_6>0$ such that for all $u>0$, $r>0$, and  $x\in \Om$,
\[
  \hat P_x\left[0\le\inf_{t\in[0,u]}q(\om_t)\le r\right]\le C_6(1+u^{-1/2})r.
\]
\end{prp}
\begin{proof}
First, let $x\in\Om\setminus \Om_r$.
From Proposition~\ref{prop:1} and the fact that $0<q(x)\le r$, 
\[
\hat P_x\left[0\le\inf_{t\in[0,u]}q(\om_t)\le r\right]
\le C_1(1+u^{-1/2})q(x)
\le C_1(1+u^{-1/2})r.
\]
Next, let $x\in\Om_r$. Then,
{\allowdisplaybreaks
\begin{align*}
&\hat P_x\left[0\le\inf_{t\in[0,u]}q(\om_t)\le r\right]
= \hat P_x\left[\tau_r\le u,\ 0\le\inf_{t\in[\tau_r,u]}q(\om_t)\right]\\
&= \hat P_x\left[\tau_r\le u,\ 0\le\inf_{t\in[0,u-\tau_r]}q((\theta_{\tau_r}\om)_t)\right]\\
&= \sum_{k=1}^\infty \hat P_x\left[2^{-k}u<u-\tau_r\le 2^{-k+1}u,\ 0\le\inf_{t\in[0,u-\tau_r]}q((\theta_{\tau_r}\om)_t)\right]\\
&\le \sum_{k=1}^\infty \hat P_x\left[2^{-k}u<u-\tau_r\le 2^{-k+1}u,\ 0\le\inf_{t\in[0,2^{-k}u]}q((\theta_{\tau_r}\om)_t)\right].
\end{align*}%
}%
Here, we used $\hat P_x[\tau_r=u]=0$ in the third line, which follows from Lemma~\ref{lem:differentiable}.
From the strong Markov property and Proposition~\ref{prop:1},
\begin{align*}
&\hat P_x\left[\left.2^{-k}u<u-\tau_r\le 2^{-k+1}u,\ 0\le\inf_{t\in[0,2^{-k}u]}q((\theta_{\tau_r}\om)_t)\right|\hat\cF_{\tau_r}\right]\\
&=1_{\{2^{-k}u<u-\tau_r\le 2^{-k+1}u\}}\cdot \hat P_{\om_{\tau_r}}\!\left[0\le\inf_{t\in[0,2^{-k}u]}q(\om_t)\right]\\
&\le C_1 1_{\{2^{-k}u<u-\tau_r\le 2^{-k+1}u\}}\cdot (1+(2^{-k}u)^{-1/2})r\\
&\le C_1 1_{\{2^{-k}u<u-\tau_r\le 2^{-k+1}u\}}\cdot (1+((u-\tau_r)/2)^{-1/2})r.
\end{align*}
Therefore, 
{\allowdisplaybreaks
\begin{align*}
&\hat P_x\left[0\le\inf_{t\in[0,u]}q(\om_t)\le r\right]\\
&\le C_1 r\hat E_x[1+((u-\tau_r)/2)^{-1/2};\ \tau_r\le u]\\
&\le C_1 r\int_0^u \left(1+\left(\frac{u-s}2\right)^{-1/2}\right)\,\hat P_x^r(ds)\\
&\le C_1 r(1+(u/4)^{-1/2})\hat P_x[\tau_r\le u/2]
+C_1 r\int_{u/2}^u \left(1+\left(\frac{u-s}2\right)^{-1/2}\right)C_2s^{-1}\,ds\\
&\le C_1 r(1+(u/4)^{-1/2})+\frac{2C_1C_2 r}{u}\int_{u/2}^u\left(1+\left(\frac{u-s}2\right)^{-1/2}\right)\,ds\\
&\le C_6(1+u^{-1/2})r.
\end{align*}%
}%
Here, we used Lemma~\ref{lem:differentiable} in the third line.
This completes the proof.
\end{proof}
Let $\hat P_{a,b}$ be a probability measure on $W$ such that $\{\om_t\}_{t\in[0,1]}$ is a pinned Brownian motion under $\hat P_{a,b}$ with $\om_0=a$ and $\om_1=b$.
The following lemma is proved by the definition of $\hat P_{a,b}$ and the monotone class theorem.
\begin{lmm}\label{lem:A}
For $t\in [0,1)$, $A\in \hat\cF_t$, and a Borel set $D$ of $\R^d$,
\[
  \hat P_{a,b}[A\cap \{\om_t\in D\}]
  \le \hat P_{a}[A\cap \{\om_t\in D\}]
  \cdot \sup_{y\in D}\frac{p_{1-t}(y,b)}{p_1(a,b)}.
\]
\end{lmm}
\begin{lmm}\label{lem:B}
Let $\tau$ be an $\{\hat\cF_t\}$-stopping time and $A\in \hat\cF_\tau$.
Let $D$ be an open set of $\R^d$. 
Then,
\[
\hat P_{a,b}[\{\tau<1\}\cap A\cap\{\om_\tau\in D\}]
\le \hat P_a[\{\tau<1\}\cap A\cap\{\om_\tau\in \overline{D}\}]
\cdot \sup_{t\in(0,1],\,y\in D}\frac{p_t(y,b)}{p_1(a,b)}.
\]
Here, $\overline{D}$ is a closure of $D$ in $\R^d$.
\end{lmm}
\begin{proof}
Consider a sequence of $\{\hat\cF_t\}$-stopping times $\{\tau_n\}$ such that each $\tau_n$ takes only finite values of $\{t_n^{(k)}\}_{k\in\Lambda_n}$ and $\tau_n\downarrow\tau$.
Here, $\Lambda_n$ is an index set consisting of finite elements.
Then,
\begin{align*}
\{\tau<1\}\cap A\cap\{\om_{\tau}\in D\}
&\subset\liminf_{n\to\infty}
(\{\tau_n<1\}\cap A\cap\{\om_{\tau_n}\in D\})\\
&\subset\limsup_{n\to\infty}
(\{\tau_n<1\}\cap A\cap\{\om_{\tau_n}\in D\})\\
&\subset \{\tau<1\}\cap A\cap\{\om_{\tau}\in \overline{D}\}.
\end{align*}
For $t\in[0,1)$, $\{\tau_n=t\}\cap A\cap \{\om_{\tau_n}\in D\}\in\hat \cF_t$.
Therefore, from Lemma~\ref{lem:A},
\begin{align*}
&\hat P_{a,b}[\{\tau_n<1\}\cap A\cap \{\om_{\tau_n}\in D\}]\\
&=\sum_{k\in \Lambda_n,\,t_n^{(k)}<1}\hat P_{a,b}[\{\tau_n=t_n^{(k)}\}\cap A\cap\{\om_{\tau_n}\in D\}]\\
&\le\sum_{k\in \Lambda_n,\,t_n^{(k)}<1}\hat P_{a}[\{\tau_n=t_n^{(k)}\}\cap A\cap\{\om_{\tau_n}\in D\}]
\cdot\sup_{y\in D}\frac{p_{1-t_n^{(k)}}(y,b)}{p_1(a,b)}\\
&\le\sum_{k\in \Lambda_n,\,t_n^{(k)}<1}\hat P_{a}[\{\tau_n=t_n^{(k)}\}\cap A\cap\{\om_{\tau_n}\in D\}]
\cdot\sup_{t\in(0,1],\,y\in D}\frac{p_t(y,b)}{p_1(a,b)}\\
&=\hat P_a[\{\tau_n<1\}\cap A\cap\{\om_{\tau_n}\in D\}]
\cdot\sup_{t\in(0,1],\,y\in D}\frac{p_t(y,b)}{p_1(a,b)}.
\end{align*}
By letting $n\to\infty$, we complete the proof from Fatou's lemma.
\end{proof}
Denote the Borel $\sg$-field on $[0,\infty)$ by ${\cal B}([0,\infty))$.
\begin{lmm}\label{lem:C}
Let $\tau$ be an $\{\hat\cF_t\}$-stopping time such that $\tau\le1$
and $\Gm\subset [0,\infty)\times W$ an element in ${\cal B}([0,\infty))\otimes\hat\cF_\infty$.
Assume that 
\[
\{\tau<1\}\cap\{((1-\tau)/2,\theta_\tau\om)\in\Gm\}\in \hat\cF_{(1+\tau)/2}.
\]
Then, for a Borel set $D$ of $\R^d$,
\begin{align*}
&\hat P_{a,b}[\{\tau<1\}\cap\{\om_\tau\in D\}\cap
\{((1-\tau)/2,\theta_\tau\om)\in\Gm\}]\\
&\le \sup_{x\in D}\hat E_x\left[\sup_{s\in(0,1/2]}1_\Gm(s,\om)\cdot s^{-d/2}\right]\exp(|a-b|_{\R^d}^2/2).
\end{align*}
\end{lmm}
\begin{proof}
Let $c\in(0,1)$. Then,
\begin{align*}
&\hat P_{a,b}[\{\tau<1\}\cap\{\om_\tau\in D\}\cap
\{((1-\tau)/2,\theta_\tau\om)\in\Gm\}]\\
&=\sum_{k=0}^\infty\hat P_{a,b}[\{c^{k+1}<1-\tau\le c^k\}\cap\{\om_\tau\in D\}\cap
\{((1-\tau)/2,\theta_\tau\om)\in\Gm\}].
\end{align*}
Since $c^{k+1}<1-\tau\le c^k$ implies that $\tau< 1-c^{k+1}$ and $(1+\tau)/2<1-c^{k+1}/2$, by combining the assumption,
the set in $\hat P_{a,b}[{}\cdots{}]$ belongs to $\hat\cF_{1-c^{k+1}/2}$.
From Lemma~\ref{lem:A} and the strong Markov property, the above equation is dominated by
\begin{align*}
&\sum_{k=0}^\infty\hat P_a[\{c^{k+1}<1-\tau\le c^k\}\cap\{\om_\tau\in D\}\cap
\{((1-\tau)/2,\theta_\tau\om)\in\Gm\}]\\
&\phantom{\sum_{k=0}^\infty}\cdot p_1(a,b)^{-1}(\pi c^{k+1})^{-d/2}\\
&\le \hat E_a[1_{\{\tau<1\}\cap\{\om_\tau\in D\}\cap
\{((1-\tau)/2,\theta_\tau\om)\in\Gm\}}\cdot (\pi(1-\tau)c)^{-d/2}]\cdot p_1(a,b)^{-1}\\
&\le \hat E_a\!\left[1_{\{\tau<1\}\cap\{\om_\tau\in D\}}
 \sup_{s\in(0,1/2]}\left(1_\Gm(s,\theta_\tau\om)\cdot (2\pi sc)^{-d/2}\right)\right]\cdot p_1(a,b)^{-1}\\
&= \hat E_a\!\left[1_{\{\tau<1\}\cap\{\om_\tau\in D\}}
\hat E_{\om_\tau}\!\left[\sup_{s\in(0,1/2]}1_\Gm(s,\om)\cdot (2\pi sc)^{-d/2}\right]\right]\cdot p_1(a,b)^{-1}\\
&\le \sup_{x\in D}
\hat E_x\!\left[\sup_{s\in(0,1/2]}1_\Gm(s,\om)\cdot s^{-d/2}\right]\cdot c^{-d/2}\exp(|a-b|_{\R^d}^2/2).
\end{align*}
By letting $c\to1$, we reach the conclusion.
\end{proof}
\begin{prp}\label{prop:2}
  There exists $C_7>0$ such that for every $r>0$,
  \begin{equation}\label{eq:r}
  \mu_{a,b}\left[w\in W_{a,b}\left|\;0\le\inf_{t\in[0,1]}q(w(t))\le r\right]\right.
  \le C_7 r.
  \end{equation}
\end{prp}
\begin{proof}
  Let $\alpha=\min\{q(a),q(b)\}/2$.
  It is sufficient to prove that there exists $C_7>0$ such that \eqref{eq:r} holds for all $r\in(0,\alpha/3)$.
  Choose $r\in(0,\alpha/3)$ and let $V=B(b,\alpha)$ and $V'=B(b,\alpha/2)$. Then,
  \begin{align*}
  &\mu_{a,b}\left[w\in W_{a,b}\left|\ 0\le\inf_{t\in[0,1]}q(w(t))\le r\right.\right]\\
  &\le \hat P_{a,b}[\tau_r<1\text{ and }\om_t\in\overline{\Om}\setminus V \text{ for all }t\in[\tau_r,(\tau_r+1)/2]]\\
  &\phantom{{}\le{}}{}+ \hat P_{a,b}\left[\begin{array}{l}\tau_r<1,\ \om_t\in V \text{ for some }t\in[\tau_r,(\tau_r+1)/2],\\
  \text{and } 
  \om_t\in\overline{\Om} \text{ for all }t\in[\tau_r,(\tau_r+1)/2]
  \end{array}
  \right]\\
  &=:I_1+I_2.
\end{align*}
For $I_1$, Lemma~\ref{lem:B} with $\tau=\min\{(\tau_r+1)/2, 1\}$ implies
\begin{align*}
  I_1&= \hat P_{a,b}[\tau_r<1,\ \om_t\in\overline{\Om}\setminus V \text{ for all }t\in[\tau_r,\tau],\ \text{and }\om_\tau\in\R^d\setminus V']\\
  &\le\hat P_a[\tau_r<1,\ \om_t\in\overline{\Om}\setminus V \text{ for all }t\in[\tau_r,\tau]]\cdot \sup_{t\in(0,1],\ y\in\R^d\setminus V'}\frac{p_t(y,b)}{p_1(a,b)}.
\end{align*}
Now,
  \[
  \sup_{t\in(0,1],\ y\in\R^d\setminus V'}\frac{p_t(y,b)}{p_1(a,b)}
  \le p_1(a,b)^{-1}\sup_{t\in(0,1]}(2\pi t)^{-d/2}\exp\left(-\frac{(\alpha/2)^2}{2t}\right)
  \le C_8
  \]
  and
  \begin{align*}
  &\hat P_a[\tau_r<1,\ \om_t\in\overline{\Om}\setminus V \text{ for all }t\in[\tau_r,\tau]]\\
  &\le \hat P_a[\tau_r<1,\ (\theta_{\tau_r} \om)_t\in\overline{\Om} \text{ for all }t\in[0,(1-\tau_r)/2]]\\
  &\le \sum_{k=1}^\infty \hat P_a[2^{-k}<1-\tau_r\le 2^{-k+1},\ (\theta_{\tau_r} \om)_t\in\overline{\Om} \text{ for all }t\in[0,2^{-k-1}]].
  \end{align*}
  Since
  \begin{align*}
  &\hat P_a[2^{-k}<1-\tau_r\le 2^{-k+1},\ (\theta_{\tau_r} \om)_t\in\overline{\Om} \text{ for all }t\in[0,2^{-k-1}]\mid \hat\cF_{\tau_r}]\\
  &=1_{\{2^{-k}<1-\tau_r\le 2^{-k+1}\}}\cdot\hat P_{\om_{\tau_r}}[\om_t\in\overline{\Om} \text{ for all }t\in[0,2^{-k-1}]]\\
  &\le 1_{\{2^{-k}<1-\tau_r\le 2^{-k+1}\}}\cdot C_1(1+2^{(k+1)/2})r\\
  &\le 1_{\{2^{-k}<1-\tau_r\le 2^{-k+1}\}}\cdot C_1(1+2(1-\tau_r)^{-1/2})r
  \end{align*}
from the strong Markov property and Proposition~\ref{prop:1},
  \begin{align*}
  I_1&\le C_8C_1r \hat E_a[1_{\{\tau_r<1\}}\cdot(1+2(1-\tau_r)^{-1/2})]\\
  &\le C_8C_1 r(1+2\sqrt2)\hat P_a[\tau_r\le 1/2]
  +C_8C_1 r \int_{1/2}^1 (1+2(1-s)^{-1/2})\cdot C_2 s^{-1}\,ds\\
  &\le C_9 r,
  \end{align*}
  by virtue of Lemma~\ref{lem:differentiable}.
  
  We will estimate a value for $I_2$.
  From Lemma~\ref{lem:C} with $\tau=\min\{\tau_r,1\}$, $D=\partial\Om_r$, and
  \[
  \Gm=\left\{(s,\om)\in[0,\infty)\times W\left|\;
  \begin{array}{l}
  \om_t\in V \text{ for some }t\in[0,s]\text{ and}\\
  \om_t\in\overline{\Om} \text{ for all }t\in[0,s]
  \end{array}\right.\right\},
  \]
  we obtain
  \[
  I_2\le C_{10} \sup_{x\in \partial\Om_r}\hat E_x\left[\sup_{s\in(0,1/2]}1_\Gm(s,\om)\cdot s^{-d/2}\right].
  \]
  By letting $f(\om)=\sup_{s\in(0,1/2]}1_\Gm(s,\om)\cdot s^{-d/2}$, we have
  \begin{equation}\label{eq:I2}
  I_2
  \le C_{10}\sup_{x\in\partial\Om_r}\hat E_x[f]
  =C_{10} \sup_{x\in\partial\Om_r}\int_0^\infty \hat P_x[f> u]\,du.
  \end{equation}
  Let $x\in\partial\Om_r$, and define $y$, $z$, $K$, $R_t$, and $S_t$ as in Lemma~\ref{lem:ito}.
  We have $|x-z|_{\R^d}=\dl+r\in(\dl,\dl+\alpha)$ and $|b-z|_{\R^d}\ge\dl+q(b)\ge\dl+2\alpha$.
  Define the stopping times with respect to the canonical augmentation of $\{\hat\cF_t\}$ by $\{\hat P_x\}_{x\in\R^d}$ as follows:
  \begin{align*}
  \rho&=\inf\{t\ge0\mid R_t\notin [\dl, \dl+\alpha)\},\\
  \rho'&=\inf\{t\ge0\mid r+\dl+Kt+S_t\notin [\dl,\dl+\alpha)\}.
  \end{align*}
  Since $\om_t\in V$ implies $R_t\ge|b-z|_{\R^d}-\alpha\ge\dl+\alpha$,
  \begin{align*}
  &\hat P_x\left[f> u\right]\\
  &\le\hat P_x[(s,\om)\in\Gm \text{ for some }s< u^{-2/d}]\\
  &\le \hat P_x[\rho< u^{-2/d}\text{ and }R_\rho=\dl+\alpha]\\
  &\le \hat P_x[\rho'< u^{-2/d}\text{ and }r+\dl+K\rho'+S_{\rho'}=\dl+\alpha]
  \quad\text{(from Lemma~\ref{lem:ito})}\\
  &\le \hat E_x[\exp(1-u^{2/d}\rho');\, r+K\rho'+S_{\rho'}=\alpha]\\
  &=e^{1+K(\alpha-r)}{\sinh\left(r\sqrt{2u^{2/d}+K^2}\right)}\left/{\sinh\left(\alpha\sqrt{2u^{2/d}+K^2}\right)}\right..
  \quad \text{(cf.\ \cite[p.~309]{BS})}
  \end{align*}
  Since $(v/4)e^{v/2}\le \sinh v\le v e^v$ for $v\ge0$, the above term is dominated by
  \begin{align*}
  &\frac{e^{1+K(\alpha-r)}r\sqrt{2u^{2/d}+K^2}\exp\left(r\sqrt{2u^{2/d}+K^2}\right)}{(\alpha/4)\sqrt{2u^{2/d}+K^2}\exp\left(\alpha\sqrt{2u^{2/d}+K^2}/2\right)}\\
  &=4\alpha^{-1}e^{1+K(\alpha-r)}r\exp\left((r-\alpha/2)\sqrt{2u^{2/d}+K^2}\right)\\
  &\le 4\alpha^{-1}e^{1+K\alpha}r\exp(-\sqrt{2}\alpha u^{1/d}/6).
  \end{align*}
Substituting this estimate into \eqref{eq:I2},
we obtain $I_2\le C_{11} r$.
  This completes the proof.
\end{proof}
\section{Proof of Theorem~\ref{th:main}}
In this section, we prove Theorem~\ref{th:main}.
We retain the notations in the previous sections.
We will utilize the following theorem.
\begin{thm}[(\cite{ES})]\label{th:ES}
Let $F$ be a measurable function on $W_{a,b}$ and $H_0$-Lipschitz; in other words, there exists $C>0$ such that
\[
|F(w+h)-F(w)|\le C|h|_{H_0},\quad
w\in W_{a,b},\ h\in H_0.
\]
Then, if $\int_{W_{a,b}} F^2\,d\mu_{a,b}<\infty$, $F$ belongs to $\D^{1,2}$.
Here, $\D^{1,2}$ is a first order $L^2$-Sobolev space on $W_{a,b}$ that is defined as the completion of $\cF C_b^1$ with respect to the norm $(\|\nabla\cdot\|_{L^2(W_{a,b}\to H_0,\mu_{a,b})}^2+\|\cdot\|_{L^2(W_{a,b},\mu_{a,b})}^2)^{1/2}$.
Moreover, $|\nabla F(\om)|_{H_0}\le C$ $\mu_{a,b}$-a.e.
\end{thm}

From Proposition~\ref{prop:2}, for any $r>0$,
\[
\mu_{a,b}\left[\overline{W_{a,b}^\Om}\setminus W_{a,b}^\Om\right]
=\mu_{a,b}\left[ \inf_{s\in[0,1]}
q(w(s))=0\right] 
\le C_7 r.
\]
Therefore, $\mu_{a,b}\left[\overline{W_{a,b}^\Om}\setminus W_{a,b}^\Om\right]=0$.
By combining this with the remark in \cite[p.~230]{Fu},
the bilinear form $(\cE',\cF C_b^1)$ is closable on $L^2(\overline{W_{a,b}^\Om},\mu_{a,b}|_{\overline{W_{a,b}^\Om}})$, and the closure $(\cE',\cF')$ is a quasi-regular, local, and conservative Dirichlet form.
In particular, we obtain the diffusion process $(X'_t,P'_x)$ on $\overline{W_{a,b}^\Om}$ associated with $(\cE',\cF')$.

Next, we prove that $1_{\overline{W_{a,b}^\Om}} \in BV(W_{a,b})$.
Define $F(w)=\inf_{t\in[0,1]}q(w(t))$ for $w\in W_{a,b}$.
For $n\in\N$, we define $f_n(s)=\min\{\max\{0, ns\},1\}$ for $s\in\R$ and $\rho_n(w)=f_n\left(F(w) \right)$ for $w\in W_{a,b}$.
Then, since ${W_{a,b}^\Om}=\{F(w)>0\}$, we obtain $\lim_{n\to \infty}\rho_n =1_{{W_{a,b}^\Om}}$ $\mu_{a,b}$-a.e.\ and in $L^1(W_{a,b},\mu_{a,b})$.
Therefore, from Theorem~\ref{th:1}, it is sufficient to prove
$\sup_n \|\nabla\rho_n\|_{L^1(W_{a,b}\to H_0,\mu_{a,b})} <\infty$.
We note that $q(x)$ is a Lipschitz function on $\R^d$ with Lipschitz constant 1; thus, we obtain the following estimate for $w\in W_{a,b}$ and $h\in H_0$
\begin{align*}
	|F(w+h)-F(w)| &= \left|\inf_{t\in[0,1]}q(w(t)+h(t))-\inf_{t\in[0,1]}q(w(t))\right|\\
	&\le \sup_{t\in[0,1]}|q(w(t)+h(t))-q(w(t))|\\
	&\le \sup_{t\in[0,1]}|h(t)|	
	\le |h|_{H_0}.
\end{align*}
Thus, $F$ is $H_0$-Lipschitz continuous.
From Theorem~\ref{th:ES}, we deduce that $F\in \D^{1,2}$ and $|\nabla F|_{H_0}\le 1$ $\mu_{a,b}$-a.e.

Now, we use the chain rule of $H_0$-derivative to obtain
\begin{align*}
	\|\nabla \rho_n\|_{L^1(W_{a,b}\to H_0,\mu_{a,b})} &\le \left\|n 1_{\{0\le F \le 1/n\}}|\nabla F|_{H_0}\right\|_{L^1(W_{a,b},\mu_{a,b})}\\
	&\le n\mu_{a,b}\left[0\le \inf_{t\in[0,1]}q(w(t)) \le \frac1n\right].
\end{align*}
According to Proposition~\ref{prop:2}, $\sup_n \|\nabla\rho_n\|_{L^1(W_{a,b}\to H_0,\mu_{a,b})} <\infty$.

By virtue of Theorem~\ref{th:2}, the process $X_t:=\lm_{a,b}(X'_t)$ satisfies the Skorokhod equation \eqref{eq:skorohod}.

Next, we will prove $\cp\left(\partial W_{a,b}^\Om\setminus\partial' W_{a,b}^\Om\right)=0$, where $\cp$ denotes the capacity associated with $(\cE',\cF')$.
When $w\in \partial W_{a,b}^\Om\setminus\partial' W_{a,b}^\Om$, there exist at least two points $t\in(0,1)$ such that $w(t)\in \partial\Om$.
Therefore,
\begin{align}\label{eq:cap2}
&\partial W_{a,b}^\Om\setminus\partial' W_{a,b}^\Om\\ &\subset \bigcup_{\begin{subarray}{c}0<s_1<s_2<1\\s_1,s_2\in{\Q}\end{subarray}}
\left\{ w\in W_{a,b}\left|\;\inf_{t\in[0,s_1]}q(w(t))=0,\ \inf_{t\in[s_1,s_2]}q(w(t))=0
\right\}\right.. \nonumber
\end{align}
For $\alpha,\beta \in \R$ and $s_1,s_2\in (0,1)$ with $s_1<s_2$, we define
\[A_{s_1,s_2,(\alpha ,\beta)} =\left\{w\in W_{a,b}\left|\;\inf_{t\in[0,s_1]}q(w(t))=\alpha,\ \inf_{t\in[s_1,s_2]}q(w(t))=\beta\right\}\right..\]
The right-hand side of \eqref{eq:cap2} is rewritten as $\bigcup_{0<s_1<s_2<1,\,s_1,s_2\in{\Q}} A_{s_1,s_2,(0,0)}$.
For a subset $G$ of $\R^2$, we denote $\bigcup_{(\alpha , \beta)\in G}A_{s_1,s_2,(\alpha ,\beta)}$ by $A_{s_1,s_2,G}$.

Fix $s_1,s_2\in (0,1)$ with $s_1<s_2$.
We define a map $f\colon W_{a,b}\to \R^2$ by 
\[
f(w)=\Bigl(\inf_{t\in[0,s_1]}q(w(t)),\ \inf_{t\in[s_1,s_2]}q(w(t))\Bigr).
\]
We denote an open ball in $\R^2$ with its center at $0$ and radius $r$ by $B(r)$.
By the continuity of $f$, $A_{s_1,s_2,B(r)}$ is an open neighborhood of $A_{s_1,s_2,(0,0)}$ in $W_{a,b}$.

Take $\eps >0$ and $\lambda \in (0,1)$.
We choose a smooth function $g$ on $[0,\infty)$ such that
\[
g(t)=\left\{\begin{array}{lcl}
1 &&\text{for } t \in [0,\lambda \eps),\\
\dfrac{3\log (t/{\eps})}{\log \lambda}-1 &&\text{for } t\in (\lambda^{5/9}\eps , \lambda^{4/9}\eps),\\
0 &&\text{for }t\in (\eps, \infty),
\end{array}\right.
\]
and $3(t\log \lambda)^{-1}\le g'(t) \le 0$ for all $t\in(0,\infty)$.
We define a function $\zeta\colon {\R}^2\to\R$ by $\zeta(x,y)=\sqrt{x^2+y^2}$ and set
$\iota = g\circ \zeta $.
Since $\iota\circ f$ is a bounded $H_0$-Lipschitz continuous function, it belongs to ${\D}^{1,2}$ --- to ${\cF'}$ in particular --- by virtue of Theorem~\ref{th:ES}.
Moreover, $\iota\circ f=1$ on $A_{s_1,s_2,B(\lambda \eps)}$.
Denoting the gradient operator on $\R^2$ by $\nabla_{\R^2}$, we have
{\allowdisplaybreaks
\begin{align*}
{\cE'}(\iota\circ f,\iota\circ f)
&=\frac12\int_{\overline{W_{a,b}^\Om}} |\nabla (\iota\circ f)|^2_{H_0} \, d\mu_{a,b}\\
&=\frac12\int_{\overline{W_{a,b}^\Om}} |\langle (\nabla f)(w),(\nabla_{{\R}^2} \,\iota)(f(w)) \rangle_{{\R}^2}|^2_{H_0} \, \mu_{a,b}(dw)\\
&\le C_{12}\int_{\overline{W_{a,b}^\Om}}  |(\nabla_{\R^2} \,\iota)(f(w))|_{\R^2}^2\, \mu_{a,b}(dw)\\
&=C_{12}\int_{\{(x,y)\in \R^2 \mid x\geq 0 ,y\geq 0\}}|\nabla_{\R^2} \,\iota|_{\R^2}^2\, d(f_*(\mu_{a,b}|_{\overline{W_{a,b}^\Om}}))\\
&=:I_3.
\end{align*}}%
Here, $f_*(\mu_{a,b}|_{\overline{W_{a,b}^\Om}})$ denotes the image measure of $\mu_{a,b}|_{\overline{W_{a,b}^\Om}}$ by $f$.
In the second line, $\langle \cdot,\cdot \rangle_{{\R}^2}$ denotes a pairing between the elements in $H_0\otimes \R^2$ and in $\R^2$ and has values in $H_0$.
The inequality from the second line to the third follows from the fact that $f$ is $H_0$-Lipschitz continuous.
Now,
\begin{align*}
|\nabla_{\R^2} \,\iota|_{\R^2}^2
&=\left({\partial \iota}/{\partial x}\right)^2+\left({\partial \iota}/{\partial y}\right)^2	\\
&=\left(g'\circ \zeta(x,y)\right)^2\frac{x^2}{x^2+y^2}+\left(g'\circ \zeta(x,y)\right)^2\frac{y^2}{x^2+y^2}\\
&=\left(g'\circ \zeta(x,y)\right)^2.
\end{align*}
By considering $\xi=(\zeta\circ f)_*(\mu_{a,b}|_{\overline{W_{a,b}^\Om}})$, we obtain 
\[
	I_3=C_{12}\int_{0}^{\infty}g'(r)^2\,\xi(dr) \le 9C_{12}\int_{\lambda\eps}^{\eps}(\log \lambda)^{-2}r^{-2} \,\xi(dr)=:I_4.
\]
From Lemma~\ref{lem:A}, the Markov property of $(\om_t,\hat P_x)$ and Proposition~\ref{prop:path},
{\allowdisplaybreaks
\begin{align*}
\Xi(r)&:=\xi([0,r])
= (f_*(\mu_{a,b}|_{\overline{W_{a,b}^\Om}}))(\zeta^{-1}([0,r]))\\
&=\mu_{a,b}\left[ w\in \overline{W_{a,b}^\Om} \left|\ \inf_{t\in[0,s_1]}q(w(t))^2+\inf_{t\in[s_1,s_2]}q(w(t))^2 \le r^2\right. \right]\\
&\le\mu_{a,b}\left[ w\in \overline{W_{a,b}^\Om} \left|\ 0\le\inf_{t\in[0,s_1]}q(w(t))\le r,\ 0\le\inf_{t\in[s_1,s_2]}q(w(t)) \le r\right. \right]\\
&\le \hat P_a \left[\ 0\le \inf_{t\in[0,s_1]}q(\om_t)\le r,\  0\le \inf_{t\in[s_1,s_2]}q(\om_t)\le r \right]
\cdot \frac{p_{1-s_2}(b,b)}{p_1(a,b)}\\
&=C_{13}\hat E_a\left[1_{\{0\le \inf_{t\in[0,s_1]}q(\om_t)\le r\}}\hat P_{\om_{s_1}}\left[ 0\le \inf_{t\in[0,s_2-s_1]}q(\om_t)\le r  \right]\right]\\
&\le C_{14}(1+(s_2-s_1)^{-1/2})r \hat P_a\left[ 0 \le \inf_{t\in[0,s_1]}q(\om_t)\le r\right]\\
&\le  C_{15}(1+(s_2-s_1)^{-1/2})(1+s_1^{-1/2})r^2
=C_{16}r^2.	
\end{align*}%
}%
Thus, we obtain
%{\allowdisplaybreaks
\begin{align*}
I_4&=9C_{12}\int_{\lambda\eps}^{\eps}(\log{\lambda})^{-2} \frac{1}{r^2} \,d\Xi(r)\\
&=9C_{12}(\log{\lambda})^{-2}  \left\{\left[\frac{\Xi(r)}{r^2}\right]_{\lambda\eps}^{\eps}+\int_{\lambda\eps}^{\eps}\frac{2\Xi(r)}{r^3} \,dr\right\}\\
&\le C_{17}(\log{\lambda})^{-2}  \Bigl(1+\int_{\lambda\eps}^{\eps}\frac{1}{r}  \,dr\Bigr)\\
&=C_{17}(\log{\lambda})^{-2} (1-\log {\lambda}).
\end{align*}
%}%
Therefore,
\begin{align*}
\cp(A_{s_1,s_2,(0,0)})
&\le \cp(A_{s_1,s_2,B(\lambda \eps)})\\
&\le \cE'(\iota\circ f,\iota\circ f) + \|\iota\circ f\|^2_{L^2({\overline{W_{a,b}^\Om}},\mu|_{\overline{W_{a,b}^\Om}})}\\
&\le \cE'(\iota\circ f,\iota\circ f) +\Xi(\eps)\\
&\le  C_{17}(\log{\lambda})^{-2}  (1-\log {\lambda})+C_{16}\eps^2.
\end{align*}
By letting $\eps \to 0$ and $\lambda \to 0$, we obtain $\cp(A_{s_1,s_2,(0,0)})=0$.
Therefore,
\[
\cp(\partial W_{a,b}^\Om\setminus \partial' W_{a,b}^\Om)
\le\sum_{0<s_1<s_2<1,\,s_1,s_2\in{\Q}}\cp(A_{s_1,s_2,(0,0)})=0.
\]
The last claim follows from the above result and Theorem~\ref{th:3}, and the fact that $\nu$ is a smooth measure.
This completes the proof of Theorem~\ref{th:main}.

\section{Concluding remarks}
\begin{enumerate}
\item  In a similar and simpler way, we can prove the counterpart of Theorem~\ref{th:main} that concerns the one-sided pinned path space on $\Om$; this theorem was proved in \cite{U}.
  More precisely, we define the path spaces as follows:
%{\allowdisplaybreaks
\begin{align*}
W_{a}&=\{w\in C([0,1]\to\R^d)\mid w(0)=a\},\\
W_{a}^\Om&=\{w\in C([0,1]\to \Om)\mid w(0)=a\},\\
H&=\left\{h\in C([0,1]\to\R^d)\left|\,
\begin{array}{l} 
h(0)=0,\ \text{$h$ is abslutely continuous}\\ 
\text{and } \int_0^1 |\dot h(s)|_{\R^d}^2\,ds<\infty
\end{array}\right.\right\},\\
  \partial W_{a}^\Om&=\left\{w\in W_{a}\left|\,
  \begin{array}{l}
  w(t)\in\overline{\Om}\text{ for every $t\in[0,1]$ and}\\
  w(s)\in\partial \Om\text{ for some $s\in(0,1]$}
  \end{array}
  \right.\right\},\\
  \partial' W_{a}^\Om&=\left\{w\in \partial W_{a}^\Om\left|\,
  \begin{array}{l}
  \text{there exists a unique }s\in(0,1]\\
   \text{such that } w(s)\in\partial\Om
   \end{array}\right.\right\},\\
  \overline{W_{a}^\Om}&=W_{a}^\Om\cup \partial W_{a}^\Om.
\end{align*}
%}%
Let $\mu_a$ be the probability measure on $W_a$ that is the law of the $d$-\hspace{0pt}dimensional Brownian motion starting at $a$.
Then, we can prove the claim that is modified by replacing $W_{a,b}$, $W_{a,b}^\Om$, $\mu_{a,b}$, $W_{0,0}$, $\mu_{0,0}$, $H_0$, and $\lm_{a,b}$ in Theorem~\ref{th:main} by $W_a$, $W_a^\Om$, $\mu_a$, $W_0$, $\mu_0$, $H$, and $\lm_a\colon W_a\ni w\mapsto w-a\in W_0$, respectively.
Also, from \cite[Theorem~4.4]{Hi}, $BV(W_0)\cap \bigcap_{q>1}L^q(W_0,\mu_0)\subset \D^{\alpha,p}$ if $p>1$ and $\alpha p<1$, where $\D^{\alpha,p}$ is a Sobolev space on $W_0$ with differentiability index $\alpha$ and integrability index $p$ according to the Malliavin calculus.
Therefore, we obtain the following theorem, which is a generalization of a part of the results of \cite{AZ}.
\begin{thm}
Assume that $0\in\Om$ and $\Om$ satisfies the uniform exterior ball condition. 
Then, $\mu_0(\partial W_0^\Om)=0$ and $1_{W_0^\Om}\in BV(W_0)$.
Moreover, for any real numbers $\alpha$ and $p$ such that $p>1$ and $\alpha p<1$, 
the function $1_{W_0^\Om}$ belongs to $\D^{\alpha,p}$.
\end{thm}
\item Precisely speaking, the diffusion process associated with $(\cE',\cF')$ should be called the {\em modified} reflecting Ornstein-Uhlenbeck process as in \cite{Fu,FH}, since $\cF'$ is defined as the completion of $\cF C_b^1$ and may be strictly smaller than the ``canonical'' first order $L^2$-Sobolev space $H^1(W_{a,b}^\Om)$.
When $\cF'$ is equal to $H^1(W_{a,b}^\Om)$ remains an open problem;
a partial answer is provided in \cite{Hi0}.
\end{enumerate}
%%%%%%%%%%%%%%%%%%%%%%%%%%%%%%%%%%%%%%%%%%%%%%%%%%

\end{document}